\documentclass{article}
\usepackage{amsmath,amsthm,amssymb}
\usepackage{enumerate,multicol,empheq,hyperref,xcolor}
\usepackage[numbers,sort&compress]{natbib}

\usepackage{geometry}
\usepackage[utf8]{inputenc} 
\usepackage[T1]{fontenc}    
\usepackage{hyperref}       
\usepackage{url}            
\usepackage{booktabs}       
\usepackage{amsfonts}       
\usepackage{nicefrac}       
\usepackage{microtype}      
\usepackage{lipsum}
\usepackage{graphicx}
\graphicspath{ {./images/} }

\theoremstyle{definition} 
\newtheorem{thm}{Theorem} 
\newtheorem{ex}{Example} 
\newtheorem{defn}{Definition} 

\newtheorem{remark}[thm]{Remark}
\newtheorem{lem}[thm]{Lemma}

\title{Universal Coefficients Formula for the Residual Power Series Method with General Integral Transform}

\author{
 Pisamai Kittipoom \\
  School of Mathematics\\
  Prince of Songkla University\\
  Hat Yai Songkhla Thailand, 90110 \\
  \texttt{pisamai.k@psu.ac.th} \\
}

\begin{document}
\maketitle
\begin{abstract}
This paper introduces a novel approach to address inherent limitations in the Residual Power Series (RPS) method and its variants with Laplace-like transforms when applied to solving time-fractional differential equations. Existing methods, while successful, often require computationally expensive calculations for the coefficients of the series solution.
To overcome this limitation, we propose a new framework for the RPS method that utilizes a general integral transform. This framework incorporates an explicit formula for calculating coefficients, thereby eliminating repetitive computations and streamlining the solution process. Moreover, it offers a universally applicable approach, remaining compatible with various RPS methods that employ Laplace-like transform variants.
\end{abstract}

\section{Introduction} 

Fractional differential equations (FDEs) have emerged as a powerful tool for describing various phenomena in physics, engineering, and other scientific disciplines. However, obtaining exact solutions for FDEs can be challenging. This necessitates the development of efficient and accurate methods for constructing approximate solutions.

One such technique is the Residual Power Series (RPS) method, introduced by Abu Arqub \cite{rpsm2014}. This method constructs approximate solutions as series expansions with fractional powers. At each step, the RPS method determines the coefficients of the series by solving an algebraic equation derived from the fractional derivative of the residual function. 
Shortly thereafter, Eriat et al.  \cite{lpsm2020New} developed a novel method, the Laplace Residual Power Series (LRPS) method. This approach combines the Laplace transform with the RPS method, resulting in a series solution constructed in the Laplace domain. Unlike the RPS method, the LRPS method utilizes the limit of the residual function in the Laplace domain to compute the coefficients.

The recent work by Jafari \cite{gtran2021} introduced a new general integral transform that encompasses various Laplace-like transforms, including the Aboodh transform, the Elzaki transform, the Kamal transform, the Pourreza transform, the G-transform, and the Sumudu transform. This broader framework provided a unifying concept, allowing researchers to explore the potential of the RPS method with a wider range of transforms. 
This advancement has paved the way for the integration of these transforms with the RPS framework. 
Specific techniques, including the Elzaki RPS method  \cite{elzrps-2021,elzrps-2021Hi,elzrps-2023,elzrps-2024N,elzrps-2024}, the Aboodh RPS method \cite{abrps2022,BSabrps2023,abrps2024,abrps2024KDV}, and the Sumudu RPS method \cite{sumurps2022}, combine the strengths of the traditional RPS approach with specific Laplace-like transforms. This combination allows for efficient construction of solutions to these problems.
Furthermore, Khirsariya et al. \cite{robust2024} have extended the application of the new general integral transform by Jafari \cite{gtran2021} with the RPS method. This combined approach, named the General Residual Power Series Method (GRPSM), offers a powerful tool for solving time-fractional gas dynamic and drainage equations.  

Despite the success of traditional RPS method with various Laplace-like transforms in constructing series solutions for time-fractional differential equations, a significant challenge remains. The existing methods often involve computationally complex calculations of series solution coefficients at each step of the solution process. Furthermore, a critical question regarding the consistency of coefficients obtained through the RPS method with various types of Laplace-like transforms remains unanswered. Our work addresses this limitation by proposing a novel approach that bypasses the computational burden associated with previous RPS methods and their variants. We derive an explicit formula for calculating the coefficients, eliminating the need for repetitive residual function calculations during the solution process. This not only streamlines the solution process but also offers a universally applicable approach for RPS method with Laplace-like transform variants.

This paper is structured as follows. Section 2 establishes the groundwork by introducing the fundamental definitions and properties of fractional derivatives and the new general integral transform. Section 3 introduces the framework of our proposed generalized RPS method. Section 4 highlights the key advantage of our method: the straightforward calculation of coefficients using the derived formula. Finally, Section 5 demonstrates the effectiveness of our proposed approach through illustrative examples.

\section{Preliminaries}
\begin{defn}
The Caputo's time fractional derivative of order $\alpha$ of $\psi(x,t)$ is defined as
\[ D_t^{\alpha}\psi(x,t) = I_t^{n-\alpha}\partial_t^{n} \psi(x,t), \quad n-1 < \alpha \le n, x \in I, t > \tau > 0,\]
where $I_t^{\beta}$ is the time fractional Riemann-Liouville integral operator of order $\beta$ defined as
\[
 I_t^{\beta} \psi(x,t) = \begin{cases} \frac{1}{\Gamma(\beta)} \int_0^t (t-\tau)^{\beta-1} \psi(x, \tau)\,d\tau, &\quad \beta > 0, t > \tau \ge 0 \\ \psi(x,t),&\quad \beta = 0 
\end{cases}. 
\]
\end{defn}
In the following lemmas we review some properties of fractional derivative used throughout this paper. For more detail, we refer the reader to Ref \cite{book1, book2, book3, book4}.
\begin{lem} \label{lem-fd}
For $n-1< \alpha \le n, \gamma > -1$ and $t \ge 0$, we have
\begin{enumerate}[(1)]
\item $D^{\alpha} c = 0, \quad c \in \mathbb{R}$.
\item $D^{\alpha} t^{\gamma} = \frac{\Gamma(\gamma+1)}{\Gamma(\gamma+1-\alpha)}\,t^{\gamma-\alpha}$.
\item $D^{m\alpha} t^{k\alpha} = \begin{cases} 0,&\quad k<m \\ \Gamma(k\alpha+1),&\quad k=m \\ \frac{\Gamma(k\alpha+1)}{\Gamma((k-m)\alpha+1)} t^{(k-m)\alpha},& \quad k > m \end{cases}$,\\
where\quad $D_t^{m\alpha}=\underbrace{D_t^{\alpha}D_t^{\alpha}\dots D_t^{\alpha}}_{m\; \text{times}}.$
\end{enumerate}
\end{lem}
\begin{proof}
The proof of part (1) and (2) are in Refs \cite{book1, book2, book3, book4}
\\
To prove part (3), if $k=m$, it follows from part (2) of the lemma that   $D^{k\alpha}t^{k\alpha} = \Gamma(k\alpha+1)$. 
For fixed $k \in \mathbb{N}$ and $k<m$, if $m=1$ then $D^{\alpha} t^0 = 0$. For $m=2$, $D^{2\alpha}t^{\alpha} = D^\alpha(D^\alpha t^{\alpha}) = D^{\alpha} (\Gamma(\alpha+1)) = 0$. By mathematical induction, assume that $D^{n\alpha}t^{k\alpha} = 0$ for $m=n$. Then
\[ D^{(n+1)\alpha} t^{k\alpha} = D\left( D^{n\alpha}t^{k\alpha} \right) = 0. \]
Hence, $D^{m\alpha} t^{k\alpha} = 0$ for all $k<m$. Finally, for the case $k>m$, we refer to part (2) of the lemma. 
\end{proof}
\begin{defn} (\cite{gtran2021,robust2024})
Let $f(t)$ be a integrable function defined for $t \ge 0$ and $|f(t)| \le Me^{\delta t}$ for some $\delta, M>0$. The general integral transform of $f(t)$, denoted by $T\{f(x)\}$ is defined by 
\[ T\{f(x)\}(s) = \nu(s) \int_0^{\infty} f(t) e^{-\omega(s)t}\,dt, \quad \omega(s)>\delta, \]
where $\nu(s) \ne 0$ and $\omega(s)$ are positive real functions.
\end{defn}
\begin{table}[ht!]
\centering
\begin{tabular}{lccl}
\hline 
\\[-0.8em]
Integral transform & $\omega(s)$ & $\nu(s)$ &  $\nu(s) \int_0^{\infty} f(t) e^{-\omega(s)t}\,dt$
\\[-0.8em]
\\ \hline 
\\[-0.9em]
Aboodh transform & $\frac{1}{s}$ & $1$ & $\frac{1}{s}\int_0^\infty f(t)e^{-st}\,dt$
\\[-0.9em]
\\  
Elzaki transform & $s$ & $\frac{1}{s}$ & $s\int_0^\infty f(t)e^{-\frac{t}{s}}\,dt$
\\[-0.9em]
\\  
Kamal transform & $1$ & $\frac{1}{s}$ & $\int_0^\infty f(t)e^{-\frac{t}{s}}\,dt$
\\[-0.9em]
\\
Laplace transform & $1$ & $s$ & $\int_0^\infty f(t)e^{-st}\,dt$
\\[-0.9em]
\\
$\alpha-$Laplace transform & $1$ & $s^{\frac{1}{\alpha}}$ & $\int_0^\infty f(t)e^{-\frac{t}{s^{\alpha}}}\,dt,\; \alpha \in \mathbb{R}^+$
\\[-0.8em]
\\ 
 G\_transform  & $s^\alpha$ & $\frac{1}{s}$ & $s^\alpha\int_0^\infty f(t)e^{-\frac{t}{s}}\,dt, \; \alpha \in \mathbb{Z}$
\\[-0.9em]
\\
Mohand transform & $s^2$ & $s$ & $ s^2\int_0^\infty f(t)e^{-st}\,dt$
\\[-0.8em]
\\ 
Pourreza transform & $s$ & $s^2$ & $s\int_0^\infty f(t)e^{-s^2t}\,dt$
\\[-0.9em]
\\ 
Sawi transform & $\frac{1}{s^2}$ & $\frac{1}{s}$ & $\frac{1}{s^2}\int_0^\infty f(t)e^{-\frac{t}{s}}\,dt$
\\[-0.9em]
\\ 
Sumudu transform & $\frac{1}{s}$ & $\frac{1}{s}$ & $\frac{1}{s} \int_0^\infty f(t)e^{-\frac{1}{s}}\,dt$
\\[-0.9em]
\\ \hline 
\end{tabular}
\caption{Laplace-like transforms}
\end{table}
\begin{lem} \cite{gtran2021,robust2024}\label{lem-prelim} 
 Let $f(x,t)$ be a piecewise continuous function on $I \times [0, \infty)$ and let $F(x,s)={T}\{ f(x,t) \}$ and $G(x,s)={T}\{ g(x,t) \}$ be the general transforms with respect to $t$ of $f$ and $g$, respectively. Then
\begin{enumerate}[(1)]
\item ${T}\{ c_1f(x,t)+c_2g(x,t)\} =  c_1F(x,s) + c_2 G(x,s).$
\item $T\{ t^{\alpha} \} = \dfrac{\Gamma(1+\alpha)\nu(s)}{\omega(s)^{1+\alpha}}, \quad \alpha > 0.$
\item $\lim\limits_{\omega(s) \to \infty} \dfrac{\omega(s)}{\nu(s)} F(x,s) = f(x,0), \quad x \in I$
\item ${T}_t \{ D_t^{n\alpha}f(x,t) \} = \omega(s)^{n\alpha} {T}_t\{f(x,t)\} - \nu(s)\sum\limits_{k=0}^{n-1}\omega(s)^{(n-k)\alpha-1}D_t^{k\alpha}\psi(x,0),\quad 0< \alpha <1$ \\
where $D_t^{n\alpha}=\underbrace{D_t^{\alpha}D_t^{\alpha}\dots D_t^{\alpha}}_{n\; \text{times}}.$
\end{enumerate}
\end{lem}

\section{Constructing the GRPS solution for time-fractional differential equations}
General Residue Power Series Method (GRPSM) combines the strengths of the general integral transform and the residual power series method to achieve an approximate or exact series solution for time-fractional differential equations. 
Consider the time-fractional partial differential equation of order $\alpha \in (n-1, n]$:
\begin{equation} \label{lr-1}
 D^{n\alpha}_t\psi(x,t) = \mathcal{N}[\psi(x,t)]+h(x,t) 
\end{equation} 
subject to initial conditions:
\begin{align} \label{ics}
  \psi(x,0)=f_0(x), D_t^{\alpha}\psi(x,0)=f_1(x), \dots, D_t^{(n-1)\alpha}\psi(x,0)=f_{n-1}(x),  
\end{align}  
where $\mathcal{N}[\psi(x,t)]$ is a nonlinear term of unknown function $\psi$ and its derivatives with respect to $x$ and $h(x,t)$ is a non-homogeneous function.  

In many studies, common forms of $\mathcal{N}[\psi(x,t)]$ include integer power functions or their derivatives with respect to the spatial variable $x$. Consequently, we will assume $\mathcal{N}[\psi(x,t)]$ represents an operator expressed as:  
\begin{align} \label{nonli}
\mathcal{N}[\psi(x,t)] &= \sum_{m=1}^M\sum_{n=0}^{N} c_{m}(x) \psi^{m}(x,t) D_x^{n} \psi(x,t) + h(x,t), 
\end{align} 
where $c_n(x)$ is a coefficient function of variable $x$ and $h(x,t)$ is a non-homogeneous function. 
\\
This method starts by expressing the solution to (\ref{lr-1}) as a fractional power series of the variable $t$ about the initial point $t=0$:
\begin{equation}\label{phi-1}
 \psi(x,t) = \sum_{k=0}^\infty \frac{\phi_k(x)}{\Gamma(k\alpha + 1)} t^{n\alpha}
\end{equation}
where $0< \alpha \le 1$ and $t > 0$.
\\
Suppose $|\psi(x,t)| < Me^{\delta t}$ for some $\delta, M > 0$.  By applying the general integral transform with respect to time $t$ to the fractional power series in Eq.(\ref{phi-1}), we obtain that
\begin{align}\label{phi-2}
\Psi(x,s) 
&= \sum_{k=0}^\infty \frac{\nu(s)}{\omega(s)^{k\alpha+1}}\phi_k(x),\quad \omega(s) > \delta, 
\end{align} 
where $\Psi(x,s) = T_t\{\psi(x,t)\}$.

In this method, the series solution (\ref{phi-2}) can be considered as an approximate solution to the exact solution which is obtained by utilizing the general integral transform to Eq.(\ref{lr-1}). It follow by part (4) of Lemma \ref{lem-prelim} that
\[
 \omega(s)^{n\alpha} \Psi(x,s) - \nu(s)\sum_{k=0}^{n-1}\omega(s)^{(n-k)\alpha-1} D_t^{k\alpha}\psi(x,0) = {T}_{t} \left\{ \mathcal{N}[\psi(x,t)] +  h(x,t) \right\} ,\quad \omega(s) > \delta.
\]
Therefore this result provides the exact solution to Eq.(\ref{lr-1}):
\begin{align} \label{LaEq}
\Psi(x,s) &= \sum_{k=0}^{n-1} \frac{\nu(s)}{\omega(s)^{k\alpha+1}} D_t^{k\alpha}\psi(x,0)  +\frac{1}{\omega(s)^{n\alpha}} {T}_{t} \left\{ \mathcal{N}[\psi(x,t)] + h(x,t)\right\} \nonumber 
\\
 &= \sum_{k=0}^{n-1} \frac{\nu(s)}{\omega(s)^{k\alpha+1}} f_{k}(x)  +\frac{1}{\omega(s)^{n\alpha}} {T}_{t} \left\{ \mathcal{N}[\psi(x,t)] + h(x,t)\right\}.
\end{align}
The following lemma presents the coefficient $\phi_n(x)$ in the series solutions (\ref{phi-1}) or (\ref{phi-2}). 

\begin{lem} (\cite{robust2024}) \label{lem-coef}
Let $\psi(x,t)$ be a continuous function in $I \times [0, \infty)$ and $|\psi(x,t)| \le Me^{\delta t}$ for some $\delta, M>0$. If $\Psi(x,s) = {T}_t \{ \psi(x,t)\}$ and
\[
\Psi(x,s) = \sum_{k=0}^\infty \frac{\nu(s)}{\omega(s)^{k\alpha +1}}\phi_k(x),\quad 0< \alpha \le 1, x \in I, \omega(s) > \delta, 
\]
then \quad $\phi_k(x) = D^{k\alpha}_t \psi(x,t)\,\Big|_{t=0}.$
\end{lem}
Lemma \ref{lem-coef} provides us only the first $n-1$ coefficients, related to the initial conditions (\ref{ics}): 
\[  
\phi_0(x) = \psi(x,0)=f_0(x),\; \phi_1(x) = D_t^{\alpha}\psi(x,0) = f_1(x), \dots, \phi_{n-1}(x) = D_t^{(n-1)\alpha}\psi(x,0) = f_{n-1}(x). 
\] 
Then Eq.(\ref{phi-2}) becomes
\begin{equation}\label{excHig}
\Psi(x,s) = \sum_{k=0}^{m-1} \frac{\nu(s)}{\omega(s)^{k\alpha +1}}f_k(x) + \sum_{k=m}^{\infty} \frac{\nu(s)}{\omega(s)^{k\alpha +1}}\phi_k(x).
\end{equation}

The GRPS method offers a systematic framework to determine the remaining coefficients $\phi_k(x), k=m, m+1, \dots$ within the series solution represented by Eq.(\ref{phi-2}).
\\
Let $\Psi_m(x,s)$, denoted as the $m-$approximate series solution, represent the $m-$th truncated series of $\Psi(x,s)$.
\begin{equation} \label{k-approx}
\Psi_m(x,s) = \sum_{k=0}^{n-1} \frac{\nu(s)}{\omega(s)^{k\alpha +1}}f_k(x) + \sum_{k=n}^{m} \frac{\nu(s)}{\omega(s)^{k\alpha +1}}\phi_k(x).
\end{equation} 
Analogous to the RPS method, we introduce the General residual function, denoted by $GRes(x,s)$ which quantifies the error between the approximate series solution (\ref{phi-2}) and the exact solution (\ref{excHig}) as follow:
\begin{equation} \label{resi-1}
  GRes(x,s) = \Psi(x,s) - \sum_{k=0}^{\infty} \frac{\nu(s)}{\omega(s)^{k\alpha+1}} f_{k}(x)  - \frac{1}{\omega(s)^{m\alpha}} {T}_{t} \left\{ \mathcal{N}\left[{T}_t^{-1}\{\Psi(x,s)\} \right]  + h(x,t)\right\}.
\end{equation}
The residual function for the $m-$th iteration or the $m-$th residual function $GRes_m(x,s)$ can be written as
\begin{align}
GRes_m(x,s) &= \Psi_m(x,s) - \sum_{k=0}^{m-1} \frac{\nu(s)}{\omega(s)^{k\alpha+1}} f_{k}(x) - \frac{1}{\omega(s)^{m\alpha}} {T}_{t} \left\{ \mathcal{N}\left[{T}_t^{-1}\{\Psi_m(x,s)\} \right] +h(x,t) \right\}, \label{kresi}
\end{align}
where $\Psi_k(x,s)$ is obtained from the truncated series (\ref{k-approx}). 
\\
We note that the notion of the residual function is closely related to the definition of the remainder of the series solution (\ref{phi-2}), 
\begin{align} \label{remain}
      R_n(x,s) = \Psi(x,s) - \sum_{k=0}^n \frac{\nu(s)}{\omega(s)^{k\alpha+1}}\phi_k(x). 
\end{align}
The following theorem establishes the convergence of the series solution (\ref{phi-2}).
\begin{thm}  \label{thm-res}
Let $\psi(x,t)$ be a continuous function on $I \times [0, \infty)$ and $|\psi(x,t)| \le Me^{\delta t}$ for some $\delta, M>0$. If $\Psi(x,s) = T \{ \psi(x,t)\}$ is represented as a series function in Eq. (\ref{phi-2}) and there exists a function $K$ such that $\left| \omega(s)\,T\left\{ D_t^{(n+1)\alpha} \psi(x,t)\right\} \right| \le K(x) $ on $I \times (\delta, \beta]$ where $0< \alpha \le 1$ and $\beta < \infty$, then the remainder $R_n(x,s)$ of the series (\ref{phi-2}) exhibits the following inequality 
\begin{equation} \label{reminder}
|R_n(x,s)| \le \frac{K(x)}{\omega(s)^{(n+1)\alpha+1}},\quad x \in I, \; \delta < \omega(s) \le \beta.
\end{equation}
\end{thm}
\begin{proof}  Suppose that $T\left\{ D_t^{k\alpha} \psi(x,t) \right\}$ is defined on $I \times (\delta, \beta]$ for $k=0,1,2, \dots, n$, and 
\begin{align} \label{proof}
    \left| \omega(s)\,T\!\left\{ D_t^{n\alpha} \psi(x,t)\right\} \right| \le K(x),\quad x \in I, \; \delta < \omega(s) < \beta.
\end{align}
Based on Lemma \ref{lem-coef}, the coefficient $\phi_n(x)$ of the series (\ref{phi-2}) becomes $\phi_n(x) = D_t^{n\alpha} \psi(x,0)$. Therefore, the remainder $R_n(x,s)$ in Eq.(\ref{remain}) can be written by
\begin{align*}
    R_n(x,s) &= \Psi(x,s) - \sum_{k=0}^n \frac{\nu(s)}{\omega(s)^{k\alpha+1}} D_t^{k\alpha} \psi(x,0) \\
    \omega(s)^{(n+1)\alpha+1}  R_n(x,s) &= \omega(s)^{(n+1)\alpha+1} \Psi(x,s) - \nu(s)\sum_{k=0}^n \frac{1}{\omega^{(k-n-1)\alpha}} D_t^{k\alpha} \psi(x,0) 
    \\
    &= \omega(x) \left( \omega(s)^{n\alpha} \Psi(x,s) - \nu(s)\sum_{k=0}^n \omega^{(n+1-k)\alpha-1} D_t^{k\alpha} \psi(x,0)  \right)
    \\
    &= \omega(s) T \{ D_t^{(n+1)\alpha} \psi(x,t)\}.
\end{align*}
The last equation obtained by part (4) of Lemma \ref{lem-prelim}. 
It follows by Eq.(\ref{proof}) that
\begin{align*}
    \left| R_n(x,s) \right| &< \frac{K(x)}{\omega(s)^{(n+1)\alpha+1}},\quad x \in I, \; \delta < \omega(s) < \beta.
\end{align*}
\end{proof}

\begin{remark}  \label{remm}
It follows by Eq.(\ref{reminder}) that for each $n \in \mathbb{N}$
\[
0 \le \omega(s)^{n\alpha +1} |R_n(x,s)| \le \frac{K(x)}{\omega(s)^{\alpha}}. 
\]
Thus
\[ \lim_{\omega(s) \to \infty} \omega(s)^{n\alpha +1} |R_n(x,s)| = 0. \]
Therefore
\[
 \lim_{\omega(s) \to \infty} \omega(s)^{n\alpha +1} R_n(x,s) = 0. 
\]
This is equivalent to
\begin{equation} \label{rem}
 \lim_{\omega(s) \to \infty} \omega(s)^{n\alpha +1} GRes_n(x,s) = 0,\quad n = 1,2, \dots . 
\end{equation} 
It is important to note that this limit plays a crucial role in determining the coefficients $\phi_i(x)$.
\end{remark}

\section{Procedure of Computing the Coefficients of Series Solution}
A simplification of the cumbersome coefficient calculation is lacking in most research articles. The main purpose of this study is to streamline the coefficient calculation, thereby significantly enhancing the method's efficiency. 
\subsection{GRPS method for time-fractional differential equation of order $(0,1]$}
Consider
\begin{equation} \label{lrr-1}
 D^{\alpha}_t\psi(x,t) = \mathcal{N}[\psi(x,t)]+h(x,t),\quad \psi(x,0)=f_0(x),
\end{equation} 
where $\mathcal{N}[\psi(x,t)]$ is a nonlinear term expressed in Eq.(\ref{nonli}) and  $0< \alpha \le 1$.
\\
By using the $m-$th residual functions in Eq.(\ref{kresi}), we will calculate all coefficients $\phi_m(x)$ as follows:
\\
\textbf{Step 1}  Find $\phi_1(x)$. By using Eq.(\ref{k-approx}), we obtain 
\begin{align*}
 \Psi_1(x,s)  &= \frac{\nu(s)}{\omega(s)}f_0(x) +\frac{\nu(s)}{\omega(s)^{1+\alpha}} \phi_1(x). 
\end{align*} 
By substituting $\Psi_1$ into Eq.(\ref{kresi}), we get
\begin{align*}
 GRes_1(x,s) 
 &= \frac{\nu(s)}{\omega(s)^{1+\alpha}}\phi_1(x) - \frac{1}{\omega(s)^{\alpha}} {T}_{t} \left\{ \mathcal{N} \left[ {T}_t^{-1} \left\{\frac{\nu(s)}{\omega(s)}f_0(x) +\frac{\nu(s)}{\omega(s)^{1+\alpha}} \phi_1(x) \right\} \right] + h(x,t) \right\} \\
&= \frac{\nu(s)}{\omega(s)^{1+\alpha}}\phi_1(x) - \frac{1}{\omega(s)^{\alpha}}  {T}_{t} \left\{ \mathcal{N} \left[ f_0(x)+\frac{t^{\alpha}\phi_1(x)}{\Gamma(\alpha+1)}   \right] + h(x,t) \right\}. 
\end{align*} 
By multiplying both sides of the above equation by $\frac{\omega(s)^{1+\alpha}}{\nu(s)}$ and applying limit as $\omega(s)$ tends to infinity, we obtain
\begin{align*}
\lim_{\omega(s) \to \infty} \frac{\omega(s)^{1+\alpha}}{\nu(s)} GRes_1(x,s) &=  \phi_1(x) - \lim_{\omega(s) \to \infty} \frac{\omega(s)}{\nu(s)}  {T}_{t} \left\{ \mathcal{N} \left[ f_0(x)+\frac{t^{\alpha} \phi_1(x)}{\Gamma(\alpha+1)}  \right] + h(x,t) \right\}   
\end{align*}
It follows from Eq.(\ref{rem}) that the left hand limit equals zero, we get
\[ 
\phi_1(x)  =  \lim_{\omega(s) \to \infty} \frac{\omega(s)}{\nu(s)} \,{T}_{t} \left\{ \mathcal{N} \left[ f_0(x)+\frac{t^{\alpha}\phi_1(x)}{\Gamma(\alpha+1)}   \right] + h(x,t) \right\}  \]
Using part (3) of Lemma \ref{lem-prelim}, we obtain the reduced coefficient:
\begin{equation}\label{phi}
\phi_1(x) =  \left(\mathcal{N} \left[ f_0(x)+\frac{t^{\alpha} \phi_1(x)}{\Gamma(\alpha+1)}  \right] + h(x,t) \right)\Bigg|_{t=0} = \mathcal{N}[f_0(x)]+ h(x,0). 
\end{equation}
\allowdisplaybreaks
\\
\textbf{Step 2}  Find $\phi_2(x)$. By substituting $\Psi_2$ from Eq.(\ref{k-approx}) into Eq.(\ref{kresi}),
\begin{align*}
 GRes_2(x,s)
  &= \frac{\nu(s)}{\omega(s)^{1+\alpha}}\phi_1(x) +\frac{\nu(s)}{\omega(s)^{1+2\alpha}} \phi_2(x)  \\
  &\qquad - 
  \frac{1}{\omega(s)^{\alpha}} {T}_{t} \left\{ \mathcal{N} \left[ {T}_t^{-1} 
  \left\{ \frac{\nu(s)}{\omega(s)}f_0(x)+\frac{\nu(s)}{\omega(s)^{1+\alpha}} \phi_1(x) +\frac{\nu(s)}{\omega(s)^{1+2\alpha}} \phi_2(x) \right\} \right] + h(x,t) \right\} \\
  &= \frac{\nu(s)}{\omega(s)^{1+\alpha}}\phi_1(x) +\frac{\nu(s)}{\omega(s)^{1+2\alpha}} \phi_2(x)  \\
  &\qquad - 
  \frac{\nu(s)}{\omega(s)^{\alpha}} {T}_{t} \left\{ \mathcal{N} \left[ f_0(x)+\frac{t^{\alpha}\phi_1(x)}{\Gamma(1+\alpha)}  
  +\frac{t^{2\alpha} \phi_2(x)}{\Gamma(1+2\alpha)}  \right] + h(x,t) \right\}.   
\end{align*}
Multiply the above equation by $\frac{\omega(s)^{1+2\alpha}}{\nu(s)}$ and take limit as $\omega(s)$ tends to infinity,
\begin{align*}
\lim_{\omega(s) \to \infty} \frac{\omega(s)^{1+2\alpha}}{\nu(s)} GRes_2(x,s) &= \phi_2(x) + \lim_{\omega(s) \to \infty} \Bigg( \omega(s)^{\alpha} \phi_1(x) - \frac{\omega(s)^{\alpha+1}}{\nu(s)} T_{t} \Bigg\{ \mathcal{N} \Bigg[ f_0(x)  +\frac{t^{\alpha}\phi_1(x)}{\Gamma(1+\alpha)} +\frac{t^{2\alpha} \phi_2(x)}{\Gamma(1+2\alpha)}  \Bigg]  \\
&  
 \qquad \qquad \qquad \qquad   + h(x,t) \Bigg\} \Bigg). 
\end{align*}
It follows from Eq.(\ref{rem}) that the left hand limit equals zero, we obtain
\begin{align}
&\phi_2(x) = \lim_{\omega(s) \to \infty} \left(  \frac{\omega(s)^{1+\alpha}}{\nu(s)}   {T}_{t} \left\{ \mathcal{N} \left[ f_0(x)+\frac{t^{\alpha} \phi_1(x)}{\Gamma(1+\alpha)} 
  +\frac{t^{2\alpha} \phi_2(x)}{\Gamma(1+2\alpha)}  \right] + h(x,t) \right\} - \omega(s)^{\alpha}\phi_1(x)  \right)  \label{phii2-1} \\
  &= \lim_{\omega(s) \to \infty} \frac{\omega(s)}{\nu(s)} \left( \omega(s)^{\alpha}   T_{t} \left\{ \mathcal{N} \left[ f_0(x)+\frac{t^{\alpha} \phi_1(x)}{\Gamma(1+\alpha)} 
  +\frac{t^{2\alpha} \phi_2(x)}{\Gamma(1+2\alpha)}  \right]+ h(x,t) \right\} - \omega(s)^{\alpha-1} \nu(s) \phi_1(x)  \right).  \label{18}
\end{align}
Many researchers solving RPS problems with Laplace-like transforms rely on a specific limit in Eq.(\ref{phii2-1}) to determine the coefficient $\phi_2(x)$.  However, this approach involves complex calculations and raises uncertainty about the limit's existence as $\omega(s)$ approaches infinity. This uncertainty motivates our goal of simplifying the limit.
\\
Based on the choice of $\mathcal{N}$ in Eq. (\ref{nonli}), we first observe the nonlinear term $\mathcal{N}$ in Eq. (\ref{phii2-1}) at the initial point $t=0$,
\begin{align*}
\left( \mathcal{N} \left[ f_0(x)+\frac{t^{\alpha}\phi_1(x)}{\Gamma(1+\alpha)} 
  +\frac{t^{2\alpha} \phi_2(x)}{\Gamma(1+2\alpha)}  \right] + h(x,t) \right) \Bigg|_{t=0} &= \mathcal{N} [f_0(x)] + h(x,0)= \phi_1(x) .
\end{align*}  
Substituting the above $\phi_1(x)$ into Eq.(\ref{18}) and utilizing part (4) of Lemma \ref{lem-prelim}, we get 
\begin{align*}
\phi_2(x) &= \lim_{\omega(s) \to \infty} \frac{\omega(s)}{\nu(s)} \, T_{t} \left\{ D_t^{\alpha} \left( \mathcal{N} \left[ f_0(x)+\frac{t^{\alpha} \phi_1(x)}{\Gamma(1+\alpha)} 
  +\frac{t^{2\alpha} \phi_2(x)}{\Gamma(1+2\alpha)}  \right] + h(x,t)\right)    \right\} .
\end{align*}
Now, to simplify the coefficient $\phi_2$ in the above equation, we use part (3) of Lemma \ref{lem-prelim}:
\begin{align} \label{explain-1}
\phi_2(x) &= D_t^{\alpha} \left( \mathcal{N} \left[ f_0(x)+\frac{t^{\alpha} \phi_1(x)}{\Gamma(1+\alpha)} 
  +\frac{t^{2\alpha} \phi_2(x)}{\Gamma(1+2\alpha)}  \right] + h(x,t)\right) \Bigg|_{t=0}.
\end{align}
However, the coefficient $\phi_2$ in Eq.(\ref{explain-1}) can be further simplified based on the choice of $\mathcal{N}$ in Eq. (\ref{nonli}) and part (3) of Lemma \ref{lem-fd}.
Therefore, Eq.(\ref{explain-1}) can be reduced to  
\begin{align} \label{phii2-new}
\phi_2(x) &= D_t^{\alpha} \left( \mathcal{N} \left[ f_0(x)+\frac{t^{\alpha} \phi_1(x)}{\Gamma(1+\alpha)} 
  +\frac{t^{2\alpha} \phi_2(x)}{\Gamma(1+2\alpha)}  \right] + h(x,t) \right) \Bigg|_{t=0} \nonumber \\
  &= D_t^{\alpha} \left( \mathcal{N} \left[ f_0(x)+\frac{t^{\alpha} \phi_1(x)}{\Gamma(1+\alpha)}   \right] + h(x,t) \right) \Bigg|_{t=0}.
\end{align}
Researchers previously calculated coefficients in the RPSM with Laplace-like transforms by finding the limit of their expansions. This repetitive method has been replaced with a more efficient general formula that applies to all coefficients.
\\
\textbf{Step 3} Using mathematical induction, we will prove that
\begin{align} \label{formu}
\phi_k(x)
&= D_t^{(k-1)\alpha} \left( \mathcal{N} \left[ f_0(x)+\sum_{m=1}^{k-1} \frac{t^{m\alpha}\phi_m(x)}{\Gamma(1+m\alpha)} \right] + h(x,t) \right) \Bigg|_{t=0}, \quad \text{for}\; k = 2,3,\dots 
\end{align}
is a general expression for determining the coefficients within the series solution (\ref{phi-2}) obtained through the general integral transform.
\begin{proof}
Base case established in  Step 2.   \\
To proceed with the inductive step, let us assume the following coefficients formulas hold true for the first $k-1$ coefficients:
\begin{equation} \label{induc}
\begin{aligned}
 \phi_l(x)
&= D_t^{(l-1)\alpha} \left( \mathcal{N} \left[ \psi(x,0)+\sum_{m=1}^{l-1} \frac{t^{m\alpha}\phi_m(x)}{\Gamma(1+m\alpha)} \right] + h(x,t) \right) \Bigg|_{t=0},\quad l = 2,3,\dots, k-1.
\end{aligned}
\end{equation}
With the assumption that $\mathcal{N}$ is defined as in Eq. (\ref{nonli}) together with part (3) of Lemma \ref{lem-fd}, we see that at the initial point $t=0$, the coefficients in Eq.(\ref{induc}) can also be expressed as
\begin{align}
\phi_l(x)
&= D_t^{(l-1)\alpha} \left( \mathcal{N} \left[ \psi(x,0)+\sum_{m=1}^{l+1} \frac{t^{m\alpha}\phi_m(x)}{\Gamma(1+m\alpha)} \right] + h(x,t)\right) \Bigg|_{t=0},\quad l = 2,3,\dots, k-1. \label{induc-2}
\end{align}
Our next step is to demonstrate the validity of the formula for the coefficient $\phi_{k+1}(x)$ 
\\
Similar to previous step, We use the $(k+1)-$th residual function and compute its limit, to obtain  
\begin{align}
\phi_{k+1}(x) &= \lim_{\omega(s) \to \infty} \left(  \omega(s)^{k\alpha+1} T_{t} \left\{ \mathcal{N} \left[ f_0(x)+ \sum_{m=1}^{k+1}\frac{t^{\alpha m} \phi_m(x)}{\Gamma(1+m\alpha)} \right] + h(x,t) \right\}  -  \sum_{m=1}^{k-1} \omega(s)^{(k-m)\alpha }\phi_{m+1}(x)   \right) \nonumber 
\\
 &= \lim_{\omega(s) \to \infty} \frac{\omega(s)}{\nu(s)}\left(\omega(s)^{k\alpha} T_{t} \left\{ \mathcal{N} \left[ f_0(x)+ \sum_{m=1}^{k+1}\frac{t^{\alpha m} \phi_m(x)}{\Gamma(1+m\alpha)} \right] + h(x,t) \right\}  -  \sum_{m=1}^{k-1} \omega(s)^{(k-m)\alpha -1}\phi_{m+1}(x)    \right)   \label{induc-3}
\end{align}
By placing coefficients $\phi_l, l=2,3,\dots,k$ in Eq. (\ref{induc-2}) into Eq. (\ref{induc-3}) and applying part (4) of Lemma \ref{lem-prelim}, we obtain
\begin{align*}
\phi_{k+1}(x) &= \lim_{\omega(s) \to \infty} \frac{\omega(s)}{\nu(s)} \, T_{t} \left\{ D_t^{k\alpha} \left( \mathcal{N} \left[ f_0(x)+\sum_{m=1}^{k+1} \frac{t^{m\alpha}\phi_m(x)}{\Gamma(1+m\alpha)} \right] + h(x,t) \right) \right\}.
\end{align*}
Based on part (3) of Lemma \ref{lem-prelim}, we have
\begin{align*}
\phi_{k+1}(x) = D_t^{k\alpha} \left( \mathcal{N} \left[ f_0(x)+\sum_{m=1}^{k+1} \frac{t^{m\alpha}\phi_m(x)}{\Gamma(1+m\alpha)} \right] + h(x,t) \right) \Bigg|_{t=0}.
\end{align*}
Thus, the proof is complete.
\end{proof}
\textbf{Step 4} As the final step, the series solution of Eq.(\ref{lr-1}) in the Laplace-like domain can be concluded by
\begin{align} \label{eq-sol-s}
\Psi(x,s) &= \frac{\psi(x,0)}{s} + \sum_{k=1}^\infty \frac{\phi_k(x)}{s^{\alpha k + 1}},
\end{align} 
where $\phi_k(x)$ is given by Eq.(\ref{formu}).
\\
Finally, by applying the inverse general transform to Eq.(\ref{eq-sol-s}). This solution in time domain is expressed in the following form
\begin{align} \label{eq:solu}
\psi(x,t) &= \psi(x,0) + \sum_{k=1}^\infty \frac{\phi_{k}(x)}{\Gamma(1+k\alpha)} \, t^{k\alpha}.
\end{align}
\subsection{GRPS method for the generalized time-fractional differential equations}
Consider the generalized time-fractional differential equations of the following form:
\begin{align} \label{highh}
D^{n\alpha} \psi(x,t) &= \mathcal{N}[\psi(x,t)]+h(x,t), \quad 0< \alpha \le 1,\; n\in \mathbb{Z}^+,
\end{align}
subject to the initial conditions (\ref{ics}), where $\mathcal{N}[\psi(x,t)]$ denotes an operator defined in Eq.(\ref{nonli}) and $h(x,t)$ is a non-homogeneous function. 
Based on the initial conditions (\ref{ics}), we obtain the first $n-1$ coefficients for the series solution (\ref{phi-2}). Therefore, we aim to find the subsequent coefficients $\phi_l(x),\;l=n, n+1, \dots$
\\
\textbf{Step 1}\;  Find $\phi_n(x)$. We substitute  $\Psi_{n}(x,s)$ from  Eq.(\ref{k-approx}) into Eq.(\ref{kresi}), to obtain
 \begin{align*}
 GRes_n(x,s) &= \frac{\nu(s)}{\omega(s)^{1+n\alpha}}\phi_n(x)- \frac{1}{\omega(s)^{n\alpha}} {T}_{t} \left\{ \mathcal{N}\left[ T_t^{-1} \left\{ \sum_{k=0}^{n-1} \frac{\nu(s)}{\omega^{1+k\alpha}}f_k(x) + \frac{\nu(s)}{\omega^{1+n\alpha}} \phi_n(x)  \right\} \right]+ h(x,t) \right\} \\
 &= \frac{\nu(s)}{\omega(s)^{1+n\alpha}}\phi_n(x)-\frac{1}{\omega(s)^{n\alpha}} {T}_t \left\{  \mathcal{N} \left[\sum_{k=0}^{n-1} \frac{t^{k\alpha}}{\Gamma(1+k\alpha)}f_k(x) + \frac{t^{n\alpha}}{\Gamma(1+n\alpha)} \phi_n(x) \right] + h(x,t) \right\}
 \end{align*}
After multiplying the above equation by $ \frac{\omega(s)^{1+m\alpha}}{\nu(s)}$ and take limit as $\omega(s)$ tends to infinity, we get
\begin{align*}
\phi_n(x)&= \lim_{s \to \infty} \frac{\omega(s)}{\nu(s)} {T}_t \left\{ \mathcal{N} \left[\sum_{k=0}^{n-1} \frac{t^{k\alpha}}{\Gamma(1+k\alpha)}f_k(x) + \frac{t^{n\alpha}}{\Gamma(1+n\alpha)} \phi_n(x) \right] + h(x,t) \right\}.
\end{align*}
By employing part (1) of Lemma \ref{lem-prelim} and based on the choice of $\mathcal{N}$ as in Eq. (\ref{nonli}), we get
\begin{align}\label{phi2-hig}
\phi_n(x) &= \mathcal{N} \left[ \sum_{k=0}^{n-1} \frac{t^{k\alpha}}{\Gamma(1+k\alpha)}f_k(x) + \frac{t^{n\alpha}}{\Gamma(1+n\alpha)} \phi_n(x)   \right] \Bigg|_{t=0} + h(x,0) = \mathcal{N}[f_0(x)] + h(x,0).
\end{align}
\\
\textbf{Step 2}\; Find $\phi_{n+1}(x)$. substituting $m=n+1$  into Eq.(\ref{k-approx}) and Eq.(\ref{kresi}) yield
\begin{align*}
 GRes_{n+1}(x,s) &= \frac{\nu(s)}{\omega(s)^{1+n\alpha}}\phi_n(x) +\frac{\nu(s)}{\omega(s)^{1+(n+1)\alpha}}\phi_{n+1}(x) 
 \\
 &\quad -\frac{1}{\omega(s)^{n\alpha}} T_t \left\{  \mathcal{N} \left[ T_t^{-1} \left\{ \sum_{k=0}^{n-1} \frac{f_k(x)}{s^{1+k\alpha}} + \frac{\phi_n(x)}{s^{1+n\alpha}} + \frac{\phi_{n+1}(x)}{s^{1+(n+1)\alpha}} \right\} \right] + h(x,t) \right\} \\
 &= \frac{\nu(s)}{\omega(s)^{1+n\alpha}}\phi_n(x) +\frac{\nu(s)}{\omega(s)^{1+(n+1)\alpha}}\phi_{n+1}(x) 
 \\
 &\qquad -\frac{1}{\omega(s)^{n\alpha}} T_t \left\{  \mathcal{N} \left[  \sum_{k=0}^{n-1} \frac{f_k(x)\,t^{k\alpha}}{\Gamma(1+k\alpha)} + \frac{\phi_n(x)\, t^{n\alpha}}{\Gamma(1+n\alpha)} + \frac{\phi_{n+1}(x)\,t^{(n+1)\alpha}}{\Gamma(1+(n+1)\alpha)}  \right] + h(x,t)\right\}.
 \end{align*}
Multiplying the previous equation by $\frac{\omega(s)^{1+(n+1)\alpha}}{\nu(s)}$ and taking limit as $\omega(s)$ approaches infinity to get
\begin{align}
\phi_{n+1}(x) &= \lim_{\omega(s) \to \infty} \left[ \frac{\omega(s)^{1+\alpha}}{\nu(s)} T_t \left\{ \mathcal{N}\left[ \sum_{k=0}^{n-1} \frac{f_k(x)\,t^{k\alpha}}{\Gamma(1+k\alpha)} + \frac{\phi_n(x)\, t^{n\alpha}}{\Gamma(1+n\alpha)} + \frac{\phi_{n+1}(x)\,t^{(n+1)\alpha}}{\Gamma(1+(n+1)\alpha)}  \right] + h(x,t) \right\} \right. \nonumber
\\
&\qquad \qquad \qquad  - \omega(s)^\alpha \phi_n(x) \Bigg]  \nonumber
\\
&= \lim_{\omega(s) \to \infty} \frac{\omega(s)}{\nu(s)}\left[ \omega(s)^{\alpha} T_t \left\{ \mathcal{N}\left[ \sum_{k=0}^{n-1} \frac{f_k(x)\,t^{k\alpha}}{\Gamma(1+k\alpha)} + \frac{\phi_n(x)\, t^{n\alpha}}{\Gamma(1+n\alpha)} + \frac{\phi_{n+1}(x)\,t^{(n+1)\alpha}}{\Gamma(1+(n+1)\alpha)}  \right] + h(x,t) \right\} \right. \nonumber
\\
&\qquad \qquad \qquad \qquad - \nu(s)\omega(s)^{\alpha-1} \phi_n(x) \Bigg]. \label{gres}
\end{align}
Based on the choice of $\mathcal{N}$ in Eq.(\ref{nonli}) and part (3) of Lemma \ref{lem-fd}, the coefficient $\phi_n(x)$ in Eq.(\ref{phi2-hig}) becomes
\[ \phi_{n}(x) = \mathcal{N}[f_0(x)] +h(x,0) = \mathcal{N}\left[ f_0(x)+\sum\limits_{k=1}^{n-1} \frac{f_k(x)\,t^{k\alpha}}{\Gamma(1+k\alpha)} + \frac{\phi_n(x)\, t^{n\alpha}}{\Gamma(1+n\alpha)} + \frac{\phi_{n+1}(x)\,t^{(n+1)\alpha}}{\Gamma(1+(n+1)\alpha)} \right] \Bigg|_{t=0} + h(x,0).
\] 
After substituting $\phi_{n}(x)$ into  Eq.(\ref{gres}), and leveraging part (4) of Lemma \ref{lem-prelim}, Eq.(\ref{gres}) can be reduced to a simpler form as shown below
\begin{align*}
\phi_{n+1}(x)
&= \lim_{\omega(s) \to \infty} \frac{\omega(s)}{\nu(s)} T_t \left\{ D_t^{\alpha} \left( \mathcal{N}\left[f_0(x)+\sum\limits_{k=1}^{n-1} \frac{f_k(x)\,t^{k\alpha}}{\Gamma(1+k\alpha)} + \frac{\phi_n(x)\, t^{n\alpha}}{\Gamma(1+n\alpha)} + \frac{\phi_{n+1}(x)\,t^{(n+1)\alpha}}{\Gamma(1+(n+1)\alpha)} \right]+ h(x,t)\right) \right\} .
\end{align*}
Utilizing part (3) of Lemma \ref{lem-prelim}, we get
\begin{align*}
\phi_{n+1}(x) &= D_t^{\alpha} \left( \mathcal{N}\left[ f_0(x)+\sum\limits_{k=1}^{n-1} \frac{f_k(x)\,t^{k\alpha}}{\Gamma(1+n\alpha)} + \frac{\phi_n(x)\, t^{n\alpha}}{\Gamma(1+n\alpha)} + \frac{\phi_{n+1}(x)\,t^{(n+1)\alpha}}{\Gamma(1+(n+1)\alpha)}  \right] + h(x,t)\right) \Bigg|_{t=0}
\end{align*}
Under the assumption of $\mathcal{N}$ in Eq.(\ref{nonli}) and using part (3) of Lemma \ref{lem-fd}, we can simplify the coefficient $\phi_{n+1}(x)$ to  
\begin{align} \label{phi_2}
\phi_{n+1}(x) &= D_t^{\alpha} \left( \mathcal{N}\left[ f_0(x) +\frac{f_1(x)\,t^\alpha}{\Gamma(1+\alpha)} \right] + h(x,t)\right) \Bigg|_{t=0}.
\end{align}
\\
\textbf{Step 3}\; Find $\phi_k(x),\;k=n+1, n+2, \dots$. We can skip the inductive step for brevity, as it follows the same approach presented in Step 3 of Section 4.1. Then the coefficients formula for the GRPS solution for time-fractional PDE. in Eq.(\ref{highh}) is
\begin{align} \label{phi-k-2}
\phi_k(x) &= D_t^{(k-n)\alpha} \left( \mathcal{N} \left[ \sum_{m=0}^{k-n}  \frac{\phi_{m}(x)\,t^{m\alpha}}{\Gamma(1+m\alpha)} \right] + h(x,t) \right) \Bigg|_{t=0},\quad k=n+1, n+2, \dots,
\end{align}
where $\phi_k(x)=f_k(x), k=0,1,2,\dots, n-1$ is obtained from the initial conditions (\ref{ics}).
\begin{remark}
   Suppose the term $\mathcal{N}$ on the right-hand side of Eq.(\ref{highh}) is linear, written as $\mathcal{N}[u]=L[u]$, where $L$ denotes a linear operator. In this case, the formula for the coefficients in Eq.(\ref{phi-k-2}) becomes
    \begin{align*}
\phi_k(x)
&= D_t^{(k-n)\alpha} \left( \sum_{m=0}^{k-n} L \left[\frac{\phi_m(x)t^{m\alpha}}{\Gamma(1+m\alpha)} \right] + h(x,t) \right) \Bigg|_{t=0}, \quad  k = n, n+1, n+2, \dots. 
    \end{align*}
Utilizing part (3) of Lemma \ref{lem-fd}, we obtain the coefficients formulas for a time-fractional linear PDE as follow
\begin{align}\label{lin-co}
    \phi_k(x) &= L[\phi_{k-n}(x)] + D_t^{(k-n)\alpha} h(x,t)\Big|_{t=0},\quad k=n, n+1, n+2, \dots.
\end{align}
\end{remark}
\section{Illustrative example}
This section emphasizes the efficacy and user-friendliness of our novel formula for determining GRPS coefficients. To showcase the broad applicability of our results, we present solutions for four well-established examples.

The first example focuses on the time-fractional Black-Scholes equation. Prior studies have employed the RPS method \cite{BSrpsm2019} and the Aboodh RPS method \cite{BSabrps2023} to solve this equation.  Our proposed formula offers a simple and alternative approach.
The second example demonstrates the solution for the (3+1) wave equation. We obtain the same solution achieved in previous work using the Elzaki RPS method \cite{elzrps-2024}.
The third example addresses the time-fractional biological population diffusion equation. We demonstrate that our formula produces a robust solution consistent with the study that utilized the Elzaki RPS method \cite{elzrps-2021}.
The final example delves into the non-homogeneous time-fractional gas dynamics equation. We derive a series solution using our formula. Notably, some coefficients derived using our formula differ slightly from the previous work utilizing  the GRPS method \cite{robust2024}. To address this difference,
we present the solution utilizing the direct method (omitted in \cite{robust2024}), demonstrating that the resulting coefficient aligns with the one derived from our formula.

In contrast to the traditional RPS approach reliant on Laplace-like transforms, our method utilizes the explicit coefficients formula to construct approximate solutions for these diverse problems, offering a more streamlined and potentially more efficient approach.

\begin{ex} \label{ex-bs} Consider a time-fractional Black-Scholes equation of order $0< \alpha \le 1$: \cite{BSrpsm2019,BSabrps2023}
\begin{align} \label{eq-bs}
D_t^\alpha V(y,t) &= -0.08(2\sin y)^2y^2\,V_{yy}-0.06yV_y +0.06V,
\\
\qquad V(y,0) &= \max\{y-25e^{-0.06},0\}. \nonumber
\end{align}
The expression on the RHS of this equation is linear: $L[V] = -0.08(2\sin y)^2y^2\,V_{yy}-0.06yV_y +0.06V$.
\\
According to formula (\ref{lin-co}) of the coefficients, we compute the following coefficients:  
\[
\phi_1(y) = L[V(y,0)]= L[\max\{y-25e^{-0.06},0\}]. 
\]
Therefore,
\[   \phi_1(y)=0.06(\max\{y-25e^{-0.06},0\}-y).  \]
Similarly, to determine the coefficient $\phi_2(x)$ in Eq.(\ref{lin-co}) 
\begin{align*}
\phi_2(y) &= L[\phi_1(y)]
\\
&= -0.08(2\sin y)^2 y^2 \phi''_1(y)-0.06y \phi_1'(y)+0.06\phi_1(y)
\\
 &= 0.06^2(\max\{y-25e^{-0.06},0\}-y).
\end{align*}
Likewise, the remaining coefficients $\phi_k(y)$ for $k=3,4,\dots$ can be  determined by
\begin{align}
\phi_k(y) &=  0.06^k(\max\{y-25e^{-0.06},0\}-y).
\end{align}
Finally, substitute these coefficients in Eq. (\ref{eq:solu}), we obtain the series solution for the time-fractional Black-Scholes equation (\ref{eq-bs}) below
\begin{align}
    V(y,t) &= 0.06(\max\{y-25e^{-0.06},0\}-y) + \sum_{n=1}^\infty 0.06^n(\max\{y-25e^{-0.06},0\}-y) \frac{t^{n\alpha}}{\Gamma(1+n\alpha)}  
\end{align}
Our solution aligns with both \cite{BSrpsm2019} (using the traditional RPS method) and \cite{BSabrps2023} (using the Aboodh RPS method), but with a more efficient approach for deriving the coefficients.
\end{ex}
\begin{ex}
Consider the $(3+1)$ wave problem: \cite{elzrps-2024}
\begin{align} \label{eq-30}
   D_t^{2\alpha} \psi(x,y,z,t) &= \frac{x^2}{2}\psi_{xx}(x,y,z,t) + \frac{y^2}{2}\psi_{yy}(x,y,z,t) +
    \frac{z^2}{2}\psi_{zz}(x,y,z,t) + x^2+y^2+z^2,\quad 0< \alpha \le 1,
\end{align}
subject to the initial conditions:
\[ \psi(x,y,z,0)=0,\quad D_t^\alpha \psi(x,y,z,0) = x^2+y^2-z^2.\]
Note that Eq.(\ref{eq-30}) is a non-homogeneous linear equation with variable coefficients, i.e.,
\[ L[\psi] = \frac{x^2}{2}\psi_{xx} + \frac{y^2}{2}\psi_{yy}+
    \frac{z^2}{2}\psi_{zz}\quad \text{and} \quad h(x,y,z,t) = x^2+y^2+z^2. \]
Applying the formula (\ref{lin-co}) for the case m=2, we obtain the following coefficients
\begin{align*}
    \phi_2(x,y,z) &= L[\psi(x,y,z,0)] + h(x,y,z,0) = L[0] + x^2+y^2+z^2 = x^2+y^2+z^2 \\
    \phi_3(x,y,z) &= L[D_t^\alpha \psi(x,y,z,0)] + D_t^{\alpha} h(x,y,z,t)\Big|_{t=0} =L[x^2+y^2-z^2] = x^2+y^2-z^2 \\
    \phi_4(x,y,z) &= L[D_t^{2\alpha} \phi_2(x,y,z)] + D_t^{2\alpha} h(x,y,z,t)\Big|_{t=0} =L[x^2+y^2+z^2] = x^2+y^2+z^2 \\   
    \phi_5(x,y,z) &= L[D_t^{3\alpha} \phi_3(x,y,z)] + D_t^{3\alpha} h(x,y,z,t)\Big|_{t=0} =L[x^2+y^2-z^2] = x^2+y^2-z^2.  
\end{align*}
This can be concluded that 
\begin{align}
    \phi_{2k} = x^2+y^2+z^2 \quad \text{and} \quad \phi_{2k-1} = x^2+y^2-z^2,\qquad \text{for}\; k \in \mathbb{Z}^+,
\end{align}
which is the same as shown in \cite{elzrps-2024} by using the Elzaki RPS method. 
\\
Furthermore, the $2k-$th approximate solution is
\[ \psi_{2k}(x,y,z,t) = (x^2+y^2+z^2) \sum_{n=1}^k \frac{t^{2n\alpha}}{\Gamma(1+2n\alpha)} + (x^2+y^2-z^2) \sum_{n=1}^k \frac{t^{(2n-1)\alpha}}{\Gamma(1+(2n-1)\alpha)}. \]
\end{ex}
\begin{ex}\label{ex-bio} Consider the time-fractional biological population diffusion equation:  \cite{elzrps-2021}
\begin{align} \label{eq-bio}
    D_t^{\alpha}u(x,y,t) &= [u^2]_{xx}+[u^2]_{yy}+cu(1-ru), 
\end{align}
subject to the initial condition:
\begin{align} \label{ex-ic}
    u(x,y,0) = \phi_0(x,y)= e^{\sqrt{\frac{cr}{8}}(x+y)}.
\end{align}
where $0< \alpha \le 1$ and $c$ is a given constant.
\\
The nonlinear RHS term is $\mathcal{N}[u] = [u^2]_{xx}+[u^2]_{yy}+cu(1-ru).$
\\
By using Eq.(\ref{phi}), we obtain
\begin{align*}
     \phi_1(x,y) &= \mathcal{N}[\phi_0(x,y)] 
     \\
     &= [\phi_0^2]_{xx} + [\phi_0^2]_{yy} + c \phi_0 (1-r\phi_0) 
\end{align*}
Based on the initial function $phi_0(x,y)$ in Eq.(\ref{ex-ic}), we obtain
\begin{align} \label{ex-p1}
    \phi_1(x,y) &= c\phi_0 = ce^{\sqrt{\frac{cr}{8}}(x+y)}.
\end{align}
Utilizing the coefficient formula (\ref{formu}) when $k=2$, we obtain the second coefficient
\begin{align}\label{eq-31}
    \phi_2(x,y) &= D_t^{\alpha} \left( \mathcal{N}\left[ \underbrace{ \phi_0+  \frac{\phi_1 t^\alpha}{\Gamma(1+\alpha)}}_{=:\Theta_1} \right]   \right) \Bigg|_{t=0} = D_t^{\alpha} \left( [\Theta_1^2]_{xx} + [\Theta_1^2]_{yy} + c \Theta_1 -cr\Theta_1^2) \right) \Big|_{t=0},
\end{align}
We now compute the following terms by using part (3) of Lemma \ref{lem-fd}:
\begin{align} \label{eq-3-2a}
D_t^{\alpha}(\Theta_1)\Big|_{t=0} = D_t^{\alpha} \left( \phi_0+  \frac{\phi_1 t^\alpha}{\Gamma(1+\alpha)} \right) \Bigg|_{t=0} = \phi_1     
\end{align}
and
\begin{align} \label{eq-3-2b}
    D_t^{\alpha}(\Theta_1^2)\Big|_{t=0} &= D_t^{\alpha} \left( \phi_0^2 + 2 \frac{\phi_0\phi_1 t^\alpha}{\Gamma(1+\alpha)}
    + \frac{\phi_1^2 t^{2\alpha}}{(\Gamma(1+\alpha))^2}\right) \Bigg|_{t=0} = 2\phi_0\phi_1.
\end{align}
Placing Eqs.(\ref{eq-3-2a}) and (\ref{eq-3-2b}) into Eq.(\ref{eq-31}),
we obtain the second coefficients
\begin{align} \label{ex4-ph2}
    \phi_2(x,y) &= 2[\phi_0\phi_1]_{xx} + 2[\phi_0\phi_1]_{yy} + c \phi_1 - 2cr \phi_0\phi_1.
\end{align}
By substituting $\phi_0$ from Eq.(\ref{ex-ic}) and $\phi_1$ from Eq.(\ref{ex-p1}) into Eq. (\ref{ex4-ph2}), we obtain the second coefficient as shown below:
\begin{align} \label{ex3eq2}
    \phi_2(x,y) &= c^2e^{\sqrt{\frac{cr}{8}}(x+y)}.
\end{align}
To determine subsequent coefficients $\phi_k(x,y), \; k= 3,4,\dots$, we use the coefficients formula (\ref{formu})
\begin{align}\label{eq-3k}
    \phi_k(x,y) &= D_t^{(k-1)\alpha} \left( \mathcal{N}\left[ \underbrace{ \sum_{n=0}^{k-1} \frac{\phi_n(x,y) t^{n\alpha}}{\Gamma(1+n\alpha)} }_{=:\Theta_{k-1}} \right]   \right)\; \Bigg|_{t=0} \nonumber
    \\
  &= D_t^{(k-1)\alpha} \left( [\Theta_{k-1}^2]_{xx} + [\Theta_{k-1}^2]_{yy} + c \Theta_{k-1} - cr\Theta_{k-1}^2  \right) \Big|_{t=0}.
\end{align}
Using part (3) of Lemma \ref{lem-fd}, we compute the following terms 
\begin{align} \label{eq-3-5a}
D_t^{(k-1)\alpha}(\Theta_{k-1})\Big|_{t=0} = D_t^{(k-1)\alpha} \left(  \sum_{n=0}^{k-1} \frac{\phi_n(x,y) t^{n\alpha}}{\Gamma(1+n\alpha)} \right) \Bigg|_{t=0} = \phi_{k-1}(x,y)  
\end{align}
and
\begin{align} \label{eq-3-5b}
    D_t^{(k-1)\alpha}(\Theta_{k-1}^2)\Big|_{t=0} &=  D_t^{(k-1)\alpha} \left( \sum_{m=0}^{k-1}\sum_{n=0}^{k-1} \frac{ t^{(m+n)\alpha}}{\Gamma(1+m\alpha)\Gamma(1+n\alpha)} \phi_{m}(x,y)\phi_{n}(x,y) \right)\Bigg|_{t=0}
    \nonumber
    \\
    &= \sum_{m=0}^{k-1} \frac{\Gamma(1+(k-1)\alpha)}{\Gamma(1+m\alpha)\Gamma(1+(k-1-m)\alpha))}\phi_m(x,y)\phi_{k-1-m}(x,y). 
\end{align}
Placing Eqs.(\ref{eq-3-5a}) and (\ref{eq-3-5b}) into Eq.(\ref{eq-3k}),
to obtain 
\begin{align}\label{eq-355}
    \phi_k(x,y) &=  \sum_{m=0}^{k-1} \frac{\Gamma(1+(k-1)\alpha)}{\Gamma(1+m\alpha)\Gamma(1+(k-1-m)\alpha))} \left( [\phi_m\phi_{k-1-m}]_{xx} + [\phi_m\phi_{k-1-m}]_{yy} -cr \phi_m\phi_{k-1-m}(x,y) \right) \nonumber
    \\
    &\qquad \qquad + c\phi_{k-1},\quad k=3,4,\dots.
\end{align}
It is easy to see that if $\phi_n(x) = c^n e^{\sqrt{\frac{cr}{8}}(x+y)}$, \; $n=0,1,\dots,k-1$, then 
\begin{align*}
     \left[\phi_m\phi_{k-1-m} \right]_{xx} + \left[\phi_m\phi_{k-1-m} \right]_{yy} -cr \phi_m\phi_{k-1-m} = 0, 
\end{align*}
for all $k=3,4,\dots$ and $m=0,1,\dots,k-1$.
\\
Hence, the $k-$th coefficient in Eq.(\ref{eq-355}) becomes
\begin{align*}
    \phi_k(x,y) &= c\phi_{k-1}(x,y), \quad k=3,4,\dots.
\end{align*}
Based on the coefficient $\psi_2$ in Eq.(\ref{ex3eq2}) , we can conclude that the coefficients of the series solution of (\ref{eq-bio}) are
\begin{align*}
    \phi_n(x,y) &= c^n e^{\sqrt{\frac{cr}{8}}(x+y)}, \quad n=0,1,2,\dots.
\end{align*}
The coefficients obtained in this work are identical to those presented in \cite{lrpsm2021soliton} for the Elzaki RPS method. However, our derivation offers a more streamlined and robust approach for determining coefficients within the context of series solutions. In particular, the $k-$th approximate solution is
\[ u_n(x,y,t) = e^{\sqrt{\frac{cr}{8}}(x+y)}\sum_{m=0}^k \frac{(ct^{\alpha})^m}{\Gamma(1+m\alpha)}. \]
As $n \to \infty$, the exact solution of the time-fractional biological population diffusion equation (\ref{eq-bio}) is
\[ u(x,y,t) = e^{\sqrt{\frac{cr}{8}}(x+y)} E_{\alpha,1}(ct^\alpha),\]
where $E_{\alpha,\beta}(t)$ is the Mittag-Leffler function defined as $E_{\alpha,\beta}(t) =\sum\limits_{m=0}^\infty \frac{t^m}{\Gamma(\beta+m\alpha)}.$
\end{ex}
\begin{ex}
Consider the non-homogeneous nonlinear time-fractional gas dynamics equation: \cite{robust2024}
\begin{align}
    D_t^{\alpha}u(x,t)=cu(1-u)-uu_x+h(x,t),\quad 0< \alpha \le 1,
\end{align}
subject to the initial condition: $u(x,0) = \phi_0(x)$, where $c$ is a given constant and 
$h(x,t)$ is a continuous function.
\\
Note that $h(x,t)$ is the non-homogeneous term and the nonlinear term is 
\begin{align} \label{ex4-nonli}
\mathcal{N}[u] =cu(1-u)-uu_x. 
\end{align}
To determine the coefficients, we use the formula (\ref{phi}) to obtain $\phi_1$: 
\begin{align} \label{ex5-ph1}
    \phi_1(x) &= \mathcal{N}[\phi_0(x)] + h(x,0) = c\phi_0(x)(1-\phi_0(x))-\phi_0(x)\phi'_0(x) + h(x,0).
\end{align}
For the subsequent coefficients, we use the formula (\ref{formu}):
\begin{align}
    \phi_2(x,t) &= D_t^{\alpha} \left( \mathcal{N}\left[\underbrace{\phi_0(x) + \frac{\phi_1(x) t^{\alpha}}{\Gamma(1+\alpha)}}_{=:\Theta_1} \right]\,\right) \;\Bigg|_{t=0} = D_t^{\alpha} \left( c\Theta_1(1 -\Theta_1) - \Theta_1[\Theta_1]_x\right) \Big|_{t=0} + D_t^{\alpha} h(x,t)\Big|_{t=0}.  \label{ex4-ph22}
\end{align}
Observe the following terms, by using part (3) of Lemma \ref{lem-fd}, we obtain 
\begin{align}
     D_t^{\alpha} \left( c\Theta_1(1 -\Theta_1)\right) \Big|_{t=0} &= c  D_t^{\alpha} \left( \Theta_1  \right)\Big|_{t=0} - c  D_t^{\alpha} \left(\Theta_1^2 \right)\Big|_{t=0} \nonumber
     \\
     &= c D_t^{\alpha} \left( \phi_0(x) + \frac{\phi_1(x) t^{\alpha}}{\Gamma(1+\alpha)} \right)\Big|_{t=0} - c D_t^{\alpha} \left( \phi_0^2 + 2 \frac{\phi_0\phi_1 t^\alpha}{\Gamma(1+\alpha)}
    + \frac{\phi_1^2 t^{2\alpha}}{(\Gamma(1+\alpha))^2}\right) \Bigg|_{t=0} 
    \nonumber
     \\
     &= c\phi_1(x) - 2c \phi_0(x) \phi_1(x), \label{A}
\end{align}
and
\begin{align}
     D_t^{\alpha} \left( \Theta_1[\Theta_1]_x\right)\Big|_{t=0} &= D_t^{\alpha} \left( \phi_0(x)\phi'_0(x) +  \frac{\phi_0(x)\phi'_1(x) \, t^{\alpha}}{\Gamma(1+\alpha)} + \frac{\phi'_0(x)\phi_1(x)\, t^{\alpha}}{\Gamma(1+\alpha)} + \frac{\phi'_1(x)\phi_1(x) t^{2\alpha}}{(\Gamma(1+\alpha))^2}   \right)\Bigg|_{t=0} \nonumber
     \\
     &= \phi_0(x)\phi'_1(x)+\phi'_0(x)\phi_1(x). \label{B}
\end{align}
Substituting Eqs.(\ref{A}) and (\ref{B}) into Eq.(\ref{ex4-ph22}), we obtain  the second coefficient 
\begin{align} \label{ex5-ph2}
    \phi_2(x) = c\phi_1(x) - 2c \phi_0(x) \phi_1(x) - \phi_0(x) \phi_1'(x) - \phi_0'(x) \phi_1(x) + D_t^{\alpha} h(x,t)\Big|_{t=0},
\end{align}
which is the same as the coefficient obtained in \cite{robust2024}.
\\
Likewise, to determine the third coefficient, we use the formula (\ref{formu}):
\begin{align}
    \phi_3(x) &=  D_t^{2\alpha} \left( \mathcal{N}\left[\underbrace{\phi_0(x) + \frac{\phi_1(x) t^{\alpha}}{\Gamma(1+\alpha)} + \frac{\phi_2(x) t^{2\alpha}}{\Gamma(1+2\alpha)}}_{=:\Theta_2} \right]\,\right) \;\Bigg|_{t=0}  \nonumber
    \\
    &= D_t^{2\alpha} \left( c\Theta_2(1 -\Theta_2) - \Theta_2[\Theta_2]_x\right) \Big|_{t=0} + D_t^{2\alpha} h(x,t)\Big|_{t=0}. \label{eq-5-0}
\end{align}
Evaluate the following terms, we use part (3) of Lemma \ref{lem-fd} to calculate fractional derivatives of the fractional power functions:
\begin{align}
    D_t^{2\alpha} (\Theta_2)|_{t=0} &= \phi_2(x), \label{eq-5-1}
    \\
    D_t^{2\alpha} (\Theta_2^2)|_{t=0} &=    D_t^{2\alpha} \left( \sum_{n=0}^2 \frac{\phi_n(x)t^{n\alpha}}{\Gamma(1+n\alpha)}\right)^2 \nonumber
    \\
    &=   D_t^{2\alpha} \left( \sum_{m=0}^2 \sum_{n=0}^2 \frac{t^{(m+n)\alpha}}{\Gamma(1+m\alpha)\Gamma(1+n\alpha)} \phi_m\phi_n \right) \Bigg|_{t=0} \nonumber
    \\
    &= \sum_{m=0}^{2} \frac{\Gamma(1+2\alpha)}{\Gamma(1+m\alpha)\Gamma(1+(2-m)\alpha))}\phi_m\phi_{2-m} \nonumber
    \\
    &= 2\phi_0\phi_2 + \frac{\Gamma(1+2\alpha)}{(\Gamma(1+\alpha))^2} \phi_1^2, \label{eq-5-2}
\end{align}
and
\begin{align}
    D_t^{2\alpha} \Big( \Theta_2[\Theta_2]_x \Big)\Big|_{t=0}
    &=   D_t^{2\alpha} \left( \sum_{m=0}^2 \sum_{n=0}^2 \frac{t^{(m+n)\alpha}}{\Gamma(1+m\alpha)\Gamma(1+n\alpha)} \phi_m\phi'_n \right) \Bigg|_{t=0}
    \nonumber \\
    &= \sum_{m=0}^{2} \frac{\Gamma(1+2\alpha)}{\Gamma(1+m\alpha)\Gamma(1+(2-m)\alpha))}\phi_m\phi'_{2-m} \nonumber
    \\
    &= \phi_0\phi'_2 + \phi'_0\phi_2 + \frac{\Gamma(1+2\alpha)}{(\Gamma(1+\alpha))^2} \phi_1\phi'_1. \label{eq-5-3}
\end{align}
Substituting Eqs.(\ref{eq-5-1}), (\ref{eq-5-2}) and (\ref{eq-5-3}) into Eq.(\ref{eq-5-0}), the third coefficient becomes
\begin{align} \label{ourph3}
    \phi_3(x) &= c\phi_2 -2c\phi_0\phi_2 -c\frac{\Gamma(1+2\alpha)}{(\Gamma(1+\alpha))^2} \phi_1^2 - \phi_0\phi'_2 - \phi'_0\phi_2 - \frac{\Gamma(1+2\alpha)}{(\Gamma(1+\alpha))^2} \phi_1\phi'_1 + D_t^{2\alpha} h(x,t)\Big|_{t=0}.
\end{align}
which is different to the third coefficient that appeared in \cite{robust2024}. 

In \cite{robust2024}, the derivation of a specific formula omitted the calculation of $\phi_3(x)$. To address this gap and ensure the validity of our own derivation, we will employ the third general residual function to directly compute $\phi_3(x)$.
\[ 
    GRes_3(x,s) = \Psi_3(x,s) - \frac{\nu(s)}{\omega(s)} \phi_0(x) - \frac{1}{\omega(s)^{\alpha}} T_t\{ \mathcal{N} [T_t^{-1} \{ \Psi_3(x,s) \}] + h(x,t)\},
\]
where 
\begin{align} \label{psii3}
     \Psi_3(x,s) = \frac{\nu(s)}{\omega(s)} \phi_0(x) + \frac{\nu(s)}{\omega(s)^{1+\alpha}} \phi_1(x) + \frac{\nu(s)}{\omega(s)^{1+2\alpha}} \phi_2(x) + \frac{\nu(s)}{\omega(s)^{1+3\alpha}} \phi_3(x). 
\end{align}
Then 
\begin{align} \label{ex-5a}
    GRes_3(x,s) &= \frac{\nu(s)}{\omega(s)^{1+\alpha}} \phi_1(x) + \frac{\nu(s)}{\omega(s)^{1+2\alpha}} \phi_2(x) + \frac{\nu(s)}{\omega(s)^{1+3\alpha}} \phi_3(x) - \frac{1}{\omega(s)^{\alpha}} T_t\{ \mathcal{N} [T_t^{-1} \{ \Psi_3(x,s) \}] + h(x,t) \}. 
\end{align}
Multiplying $\frac{\omega^{1+3\alpha}}{\nu(s)}$ to Eq. (\ref{ex-5a}) and applying limit as $\omega(s) \to \infty$ yield
\begin{align*} 
    \phi_3(x) &= \lim_{\omega(s) \to \infty} \left[ \frac{\omega(s)^{1+2\alpha}}{\nu(s)}  T_t\{ \mathcal{N} [T_t^{-1} \{ \Psi_3(x,s)\}\,]+ h(x,t)\} - \omega(s)^{2\alpha} \phi_1(x) - \omega(s)^{\alpha} \phi_2(x) \right] 
    \\
    &= \lim_{\omega(s) \to \infty} \left[ \frac{\omega(s)^{1+2\alpha}}{\nu(s)}  T_t\{ \mathcal{N} [T_t^{-1} \{ \Psi_3(x,s)\}\,] \} + \frac{\omega(s)^{1+2\alpha}}{\nu(s)}T_t\{ h(x,t)\} - \omega(s)^{2\alpha} \phi_1(x) - \omega(s)^{\alpha} \phi_2(x) \right].
\end{align*}
Substituting $\phi_1$ from Eq.(\ref{ex5-ph1}) and $\phi_2$ from Eq.(\ref{ex5-ph2}) into the above equation, to obtain
\begin{align} 
\label{ex5-ph3}
    \phi_3(x) &=  \lim_{\omega(s) \to \infty} \left[ \Big( \frac{\omega(s)^{1+2\alpha}}{\nu(s)}  T_t\{ \mathcal{N} [T_t^{-1} \{ \Psi_3(x,s)\}\,] \} -  \omega(s)^{2\alpha} \left(c\phi_0 - c\phi_0^2-\phi_0\phi'_0\right) \right. \nonumber
    \\ 
    & \qquad \qquad \qquad  - \omega(s)^{\alpha} \left( c\phi_1 - 2c\phi_0\phi_1 + (\phi_0\phi_1'+\phi_0'\phi_1) \right) \Big)
    \nonumber
    \\
    & \left.\qquad \qquad \qquad + \left(\frac{\omega(s)^{1+2\alpha}}{\nu(s)}T_t\{ h(x,t)\} -  \omega(s)^{\alpha-1} D_t^{\alpha} h(x,0)  -   \omega(s)^{2\alpha-1} h(x,0) \right)\right].
\end{align}
Based on the limit in Eq.(\ref{ex5-ph3}), let us observe these two limits:
\begin{align} \label{lim-11}
    \lim\limits_{\omega(s) \to \infty} \Big[ \frac{\omega(s)^{1+2\alpha}}{\nu(s)}  T_t\{ \mathcal{N} [T_t^{-1} \{ \Psi_3(x,s)\}\,] \} &-  \omega(s)^{2\alpha} \left(c\phi_0 - c\phi_0^2-\phi_0\phi'_0\right) \nonumber 
    \\
    &- \omega(s)^{\alpha} \left( c\phi_1 - 2c\phi_0\phi_1 + (\phi_0\phi_1'+\phi_0'\phi_1) \right) \Big],
\end{align}    
and
\begin{align} \label{lim-2}
    \lim\limits_{\omega(s) \to \infty} \left[\frac{\omega(s)^{1+2\alpha}}{\nu(s)}T_t\{ h(x,t)\} -  \omega(s)^{\alpha-1} D_t^{\alpha} h(x,0)  -   \omega(s)^{2\alpha-1} h(x,0)  \right].
\end{align}
Substituting $\Psi_3$ from Eq.(\ref{psii3}) into the limit (\ref{lim-11}) and denoting $ U =T_t^{-1} \left\{ \Psi_3(x,s) \right\}$, we then consider the term
\begin{align}
 T_t^{-1} \left\{ \Psi_3(x,s) \right\}&= T_t^{-1} \left\{ \frac{\nu(s)}{\omega(s)} \phi_0(x) + \frac{\nu(s)}{\omega(s)^{1+\alpha}} \phi_1(x) + \frac{\nu(s)}{\omega(s)^{1+2\alpha}} \phi_2(x) + \frac{\nu(s)}{\omega(s)^{1+3\alpha}} \phi_3(x) \right\}, \nonumber
    \\
 U &= \phi_0(x)+\frac{\phi_1(x) t^{\alpha}}{\Gamma(1+\alpha)} +\frac{\phi_2(x) t^{2\alpha}}{\Gamma(1+2\alpha)} +\frac{\phi_3(x) t^{3\alpha}}{\Gamma(1+3\alpha)}. \label{U}
\end{align} 
Based on $\mathcal{N}[u]$ provided in Eq.(\ref{ex4-nonli}) and the function $U$ in Eq.(\ref{U}), we observe the following term 
\begin{align*}
    T_t \left\{ \mathcal{N} \left[ T_t^{-1} \left\{ \Psi_3(x,s) \right\} \,\right]\, \right\} &= T_t \left\{ \mathcal{N}[U]\right\} 
    \\
    &=  T_t \left\{ cU-cU^2-UU_x \right\}
    \\
    &= c T_t \{U\} -c T_t\{U^2\} - T_t\{UU_x\}
    \\
    &= c T_t \left\{ \phi_0(x)+\frac{\phi_1(x) t^{\alpha}}{\Gamma(1+\alpha)} +\frac{\phi_2(x) t^{2\alpha}}{\Gamma(1+2\alpha)} +\frac{\phi_3(x) t^{3\alpha}}{\Gamma(1+3\alpha)} \right\}
    \\
    & - cT_t \left\{ \phi_0^2 +\left( \frac{\phi_1}{\Gamma(1+\alpha)}\right)^2 t^{2\alpha}  +\left( \frac{\phi_2}{\Gamma(1+2\alpha)}\right)^2 t^{4\alpha}  +\left( \frac{\phi_3}{\Gamma(1+3\alpha)}\right)^2 t^{6\alpha} \right. 
    \\
    &  \qquad \quad+ \frac{2\phi_0\phi_1t^{\alpha}}{\Gamma(1+\alpha)}  + \frac{2\phi_0\phi_2t^{2\alpha}}{\Gamma(1+2\alpha)}  + \frac{2\phi_0\phi_3t^{3\alpha}}{\Gamma(1+3\alpha)}  
    + \frac{2\phi_1\phi_2 t^{3\alpha}}{\Gamma(1+\alpha)\Gamma(1+2\alpha)}   
    \\
    & \qquad \quad + \frac{2\phi_1\phi_3 t^{4\alpha}}{\Gamma(1+\alpha)\Gamma(1+3\alpha)} + \left.\frac{2\phi_2\phi_3 t^{5\alpha}}{\Gamma(1+2\alpha)\Gamma(1+3\alpha)} \right\}
    \\
    &-  T_t \left\{ \phi_0\phi'_0 +  \frac{\phi_1\phi'_1}{(\Gamma(1+\alpha))^2} t^{2\alpha}  +  \frac{\phi_2\phi'_2} {(\Gamma(1+2\alpha))^2} t^{4\alpha}  + \frac{\phi_3 \phi'_3}{(\Gamma(1+3\alpha))^2}  t^{6\alpha} \right. 
    \\
    & \qquad \quad + \frac{(\phi_0\phi'_1+\phi'_0\phi_1)t^{\alpha}}{\Gamma(1+\alpha)}  + \frac{(\phi_0\phi'_2+\phi'_0\phi_2)t^{2\alpha}}{\Gamma(1+2\alpha)}  + \frac{(\phi_0\phi'_3+\phi'_0\phi_3)t^{3\alpha}}{\Gamma(1+3\alpha)}  
    \\
    & \qquad \quad \left. + \frac{(\phi_1\phi'_2+\phi'_1\phi_2) t^{3\alpha}}{\Gamma(1+\alpha)\Gamma(1+2\alpha)}   + \frac{(\phi_1\phi'_3+\phi'_1\phi_3) t^{4\alpha}}{\Gamma(1+\alpha)\Gamma(1+3\alpha)} +\frac{(\phi_2\phi'_3+\phi'_2\phi_3) t^{5\alpha}}{\Gamma(1+2\alpha)\Gamma(1+3\alpha)} \right\}. 
\end{align*}
Utilizing part (2) of Lemma \ref{lem-prelim}, we calculate the general transform with respect to $t$ of fractional power functions within the above equation. Subsequently, we multiply the result by $\frac{\omega(s)^{1+2\alpha}}{\nu(s)}$, to obtain 
\begin{align} \label{lim-1}
   \frac{\omega(s)^{1+2\alpha}}{\nu(s)} T_t \left\{ \mathcal{N} \left[ T_t^{-1} \left\{ \Psi_3(x,s) \right\} \,\right]\, \right\} 
    &= c \left[ \omega(s)^{2\alpha} \phi_0 + \omega(s)^{\alpha} \phi_1 +  \phi_2 + \frac{1}{\omega(s)^{\alpha}} \phi_3 \right] \nonumber
    \\
    & \quad- c \left[\phi_0^2 \omega(s)^{2\alpha} +\left( \frac{\phi_1}{\Gamma(1+\alpha)}\right)^2 \Gamma(1+2\alpha)  +\left( \frac{\phi_2}{\Gamma(1+2\alpha)}\right)^2 \frac{\Gamma(1+4\alpha)}{\omega(s)^{2\alpha}} \right. \nonumber
    \\
    & \qquad\quad  +\left( \frac{\phi_3}{\Gamma(1+3\alpha)}\right)^2 \frac{\Gamma(1+6\alpha)}{\omega(s)^{4\alpha}}  + 2\phi_0\phi_1 \omega(s)^{\alpha}  + 2\phi_0\phi_2  + \frac{2\phi_0\phi_3 }{\omega(s)^{\alpha}}  \nonumber
    \\
    & \qquad\quad +\frac{2\phi_1\phi_2\Gamma(1+3\alpha)}{\Gamma(1+\alpha)\Gamma(1+2\alpha) \omega(s)^{\alpha}}   +\frac{2\phi_1\phi_3 \Gamma(1+4\alpha)}{\Gamma(1+\alpha)\Gamma(1+3\alpha) \omega(s)^{2\alpha}}  \nonumber
    \\
    & \qquad\quad + \left.\frac{2\phi_2\phi_3 \Gamma(1+5\alpha)}{\Gamma(1+2\alpha)\Gamma(1+3\alpha) \omega(s)^{3\alpha}} \right]
    \nonumber
    \\
    &\quad - \left[ \phi_0\phi'_0 \omega(s)^{2\alpha}  +  \frac{\phi_1\phi'_1 \Gamma(1+2\alpha)}{(\Gamma(1+\alpha))^2 }   +  \frac{\phi_2\phi'_2 \Gamma(1+4\alpha)}{(\Gamma(1+2\alpha))^2 \omega(s)^{2\alpha}}      \right. \nonumber
    \\
    & \qquad \quad + \frac{\phi_3 \phi'_3 \Gamma(1+6\alpha)}{(\Gamma(1+3\alpha))^2 \omega(s)^{4\alpha}} + (\phi_0\phi'_1+\phi'_0\phi_1)\omega(s)^{\alpha}  + (\phi_0\phi'_2+\phi'_0\phi_2)  \nonumber
    \\
    & \qquad \quad + \frac{(\phi_0\phi'_3+\phi'_0\phi_3)}{\omega(s)^{\alpha}} + \frac{(\phi_1\phi'_2+\phi'_1\phi_2) \Gamma(1+3\alpha) }{\Gamma(1+\alpha)\Gamma(1+2\alpha)\omega(s)^{\alpha}}  \nonumber
    \\
    & \qquad \quad \left. + \frac{(\phi_1\phi'_3+\phi'_1\phi_3) \Gamma(1+4\alpha)}{\Gamma(1+\alpha)\Gamma(1+3\alpha) \omega(s)^{2\alpha}} +\frac{(\phi_2\phi'_3+\phi'_2\phi_3) \Gamma(1+5\alpha)}{\Gamma(1+2\alpha)\Gamma(1+3\alpha)\omega(s)^{3\alpha}} \right].
\end{align} 
Substituting the result obtained in Eq.(\ref{lim-1}) into the limit (\ref{lim-11}), we find
\begin{align} \label{lim-101}
    \lim\limits_{\omega(s) \to \infty} \Big[ &\frac{\omega(s)^{1+2\alpha}}{\nu(s)}  T_t\{ \mathcal{N} [T_t^{-1} \{ \Psi_3(x,s)\}\,] \} -  \omega(s)^{2\alpha} \left(c\phi_0 - c\phi_0^2-\phi_0\phi'_0\right) \nonumber 
    \\
    &- \omega(s)^{\alpha} \left( c\phi_1 - 2c\phi_0\phi_1 + (\phi_0\phi_1'+\phi_0'\phi_1) \right) \Big] \nonumber
    \\
    &=  c\phi_2 -c\left( \frac{\phi_1}{\Gamma(1+\alpha)}\right)^2 \Gamma(1+2\alpha) -2c\phi_0\phi_2 -\frac{\phi_1\phi'_1 \Gamma(1+2\alpha)}{(\Gamma(1+\alpha))^2 }  - (\phi_0\phi'_2+\phi'_0\phi_2)
\end{align}  
In the next step, we evaluate the limit (\ref{lim-2}) by applying parts (3) and (4) of Lemma \ref{lem-prelim} to the non-homogeneous term as shown below.
\begin{align} \label{lim-22}
    \lim\limits_{\omega(s) \to \infty} \Bigg[ &\frac{\omega(s)^{1+2\alpha}}{\nu(s)}T_t\{ h(x,t)\} -   \omega(s)^{2\alpha-1} h(x,0) -  \omega(s)^{\alpha-1} D_t^{\alpha} h(x,0) \Bigg] \nonumber
    \\
    &=\lim_{\omega(s) \to \infty}  \frac{\omega(s)}{\nu(s)}\left[ \omega(s)^{2\alpha} T_t\{ h(x,t)\} - \nu(s) \omega(s)^{2\alpha-1} h(x,0) - \nu(s)\omega(s)^{\alpha-1} D_t^{\alpha} h(x,0)  \right]  \nonumber
    \\
    &= D_t^{2\alpha} h(x,0). 
\end{align}
Upon substituting the established limits (\ref{lim-101}) and (\ref{lim-22}) into Eq.(\ref{ex5-ph3}), the third coefficient determined by the direct method is derived as follows
\begin{align*}
    \phi_3(x) &= c\phi_2 -c\left( \frac{\phi_1}{\Gamma(1+\alpha)}\right)^2 \Gamma(1+2\alpha) -2c\phi_0\phi_2 -\frac{\phi_1\phi'_1 \Gamma(1+2\alpha)}{(\Gamma(1+\alpha))^2 }  - (\phi_0\phi'_2+\phi'_0\phi_2)
      +  D_t^{2\alpha} h(x,0),
\end{align*}
which is the same as the coefficient obtained from our coefficient formula in Eq. (\ref{ourph3}). While the previous study \cite{robust2024} omitted details regarding the potentially lengthy and cumbersome direct method, our formula offers a streamlined and efficient alternative.

For the remaining coefficients $\phi_k(x), k=4,5,\dots$, we 
use the coefficients formula (\ref{formu}) to evaluate them:
\begin{align}\label{eq-4k}
    \phi_k(x) &= D_t^{(k-1)\alpha} \left( \mathcal{N}\left[ \underbrace{ \sum_{n=0}^{k-1} \frac{\phi_n(x) t^{n\alpha}}{\Gamma(1+n\alpha)} }_{=:\Theta_{k-1}} \right] + h(x,t)\,  \right)\, \Bigg|_{t=0} \nonumber
    \\
  &= D_t^{(k-1)\alpha} \left(   c\Theta_{k-1}(1 -\Theta_{k-1}) - \Theta_{k-1}[\Theta_{k-1}]_x \right) \Big|_{t=0} + D_t^{(k-1)\alpha} h(x,t)\Big|_{t=0},
\end{align}
To compute the following terms, we refer to Eqs. (\ref{eq-3-5a}) and (\ref{eq-3-5b}) in Example \ref{ex-bio} 
\begin{align} \label{eq-4-5a}
D_t^{(k-1)\alpha} \left( c\Theta_{k-1}(1 -\Theta_{k-1})\right)\Big|_{t=0} = c\phi_{k-1}(x) - c \sum_{m=0}^{k-1} \frac{\Gamma(1+(k-1)\alpha)}{\Gamma(1+m\alpha)\Gamma(1+(k-1-m)\alpha))}\phi_m\phi_{k-1-m} 
\end{align}
and
\begin{align} \label{eq-4-5b}
    D_t^{(k-1)\alpha}( \Theta_{k-1}[\Theta_{k-1}]_x )\Big|_{t=0} 
    &=  \sum_{m=0}^{k-1} \frac{\Gamma(1+(k-1)\alpha)}{\Gamma(1+m\alpha)\Gamma(1+(k-1-m)\alpha))}\phi_m\phi'_{k-1-m}. 
\end{align}
Substituting Eqs.(\ref{eq-4-5a}) and (\ref{eq-4-5b}) into Eq. (\ref{eq-4k}), we obtain  
\begin{align}\label{eq-35}
    \phi_k &=  c\phi_{k-1} - \sum_{m=0}^{k-1} \frac{\Gamma(1+(k-1)\alpha)}{\Gamma(1+m\alpha)\Gamma(1+(k-1-m)\alpha))} \left(  c\phi_m\phi_{k-1-m}  + \phi_m\phi'_{k-1-m} \right) + D_t^{(k-1)\alpha} h(x,t)\Big|_{t=0},
\end{align}
for $k=4,5,\dots$. In the following, we present the coefficients derived from the formula (\ref{eq-35}):
\begin{align*}
    \phi_4 &= c\phi_3 - 2c\phi_0\phi_3- (2c\phi_1\phi_2+\phi_1\phi'_2+\phi'_1\phi_2) \frac{\Gamma(1+3\alpha)}{\Gamma(1+\alpha)\Gamma(1+2\alpha)} - (\phi_0\phi'_3+\phi'_0\phi_3) + D_t^{3\alpha} h(x,t)\Big|_{t=0}, 
    \\
    \phi_5 &= c\phi_4 - 2c\phi_0\phi_4- (2c\phi_1\phi_3+\phi_1\phi'_3+\phi'_1\phi_3) \frac{\Gamma(1+4\alpha)}{\Gamma(1+\alpha)\Gamma(1+3\alpha)} - (c\phi_2^2+\phi_2\phi'_2)\frac{\Gamma(1+4\alpha)}{(\Gamma(1+2\alpha))^2} 
    \nonumber \\
    &\quad -  
    (\phi_0\phi'_4+\phi'_0\phi_4)+ D_t^{4\alpha} h(x,t)\Big|_{t=0},  
    \\
    \phi_6 &= c\phi_5 - 2c\phi_0\phi_5 -(2c\phi_1\phi_4+\phi_1\phi'_4 + \phi'_1\phi_4)\frac{\Gamma(1+5\alpha)}{\Gamma(1+\alpha)\Gamma(1+4\alpha)}  
    \nonumber\\
    &\quad - (2c\phi_2\phi_3+\phi_2\phi'_3 + \phi'_2\phi_3)\frac{\Gamma(1+5\alpha)}{\Gamma(1+2\alpha)\Gamma(1+3\alpha)}
    (\phi_0\phi'_5+\phi'_0\phi_5)+ D_t^{5\alpha} h(x,t)\Big|_{t=0},    
\end{align*}

The coefficient formula we propose generates results that different from those presented in \cite{robust2024}. It is noteworthy that the coefficients obtained from the direct method are identical to those derived from our proposed formula. Moreover,   this formula for coefficients offers a significant reduction in computational complexity, thereby streamlining the analytical process. 

Furthermore, we will utilize coefficients formula (\ref{eq-35}) to find the solution for the gas dynamics equation (\ref{eq-30}) with the following choices of $c$, $u(x,0)$ and $h(x,t)$:
\begin{itemize}
    \item $c=1, u(x,0)=\phi_0=e^{-x}$ and $h(x,t)=0$, the coefficients are $\phi_i(x) = e^{-x},\, i =0,1,2,\dots$, we have the final solution
    \[ u(x,t) = e^{-x} \sum_{n=0}^\infty \frac{t^{n\alpha}}{\Gamma(1+n\alpha)}. \]
   \item $c=\ln b, u(x,0)=\phi_0=b^{-x}$ and $h(x,t)=0$, the coefficients are $\phi_i(x) = b^{-x}(\ln b)^i,\, i =0,1,2,\dots$. The obtained series solution is
    \[ u(x,t) = b^{-x} \sum_{n=0}^{\infty} \frac{(\ln b)^n\,t^{n\alpha}}{\Gamma(1+n\alpha)}. \]   
   \item $c=1, u(x,0)=\phi_0=1-e^{-x}$ and $h(x,t)=-e^{-x+t}$, the coefficients are $\phi_i(x) = e^{-x},\, i =0,1,2,\dots$, we have the final solution
    \[ u(x,t) = 1-e^{-x} \sum_{n=0}^\infty \frac{t^{n\alpha}}{\Gamma(1+n\alpha)}. \]
\end{itemize}
\end{ex}

\section{Conclusion}
In conclusion, this paper has presented a novel approach to address the limitations associated with traditional RPS methods for solving time-fractional differential equations. Despite their success, these methods present significant challenges. First, calculating the coefficients for series solutions can be computationally expensive. Second, the consistency of these methods across different Laplace-like transforms remains an open question.

Our proposed method overcomes these limitations by introducing the GRPS method along with an explicit formula for coefficient calculation. This formula eliminates the need for repetitive calculations, streamlining the solution process.  Furthermore, the coefficients formula in the GRPS method offers a universally applicable approach, compatible with various RPS methods and the Laplace-like transform variants.

The paper demonstrated the effectiveness of our method through illustrative examples, showcasing its ability to efficiently and consistently solve time-fractional differential equations. This paves the way for further exploration of this approach in solving a wider range of problems involving fractional calculus.



\begin{thebibliography}{10}
\expandafter\ifx\csname url\endcsname\relax
  \def\url#1{\texttt{#1}}\fi
\expandafter\ifx\csname urlprefix\endcsname\relax\def\urlprefix{URL }\fi
\expandafter\ifx\csname href\endcsname\relax
  \def\href#1#2{#2} \def\path#1{#1}\fi

\bibitem{rpsm2014}
M.~Alquran, Analytical solutions of fractional foam drainage equation by
  residual power series method, Math Sci 8~(1) (2014) 153–160.

\bibitem{lpsm2020New}
T.~Eriqat, A.~El-Ajou, M.~Oqielat, Z.~Al-Zhour, S.~Momani, A new attractive
  analytic approach for solutions of linear and nonlinear neutral fractional
  pantograph equations, Chaos, Solitons and Fractals 138~(1) (2020) 1--11.

\bibitem{gtran2021}
H.~Jafari, A new general integral transform for solving integral equations,
  Journal of Advanced Research 32 (2021) 133--138.

\bibitem{elzrps-2021}
J.~Zhang, X.~Chen, L.~Li, C.~Zhou, Elzaki transform residual power series
  method for the fractional population diffusion equations, Engineering Letters
  29~(4) (2021) 1561--1572.

\bibitem{elzrps-2021Hi}
A.~Khan, M.~I. Liaqat, M.~Younis, A.~Alam, Approximate and exact solutions to
  fractional order cauchy reaction-diffusion equations by new combine
  techniques, Journal of Mathematics 2021 (2021) 1--12.

\bibitem{elzrps-2023}
N.~Iqbal, M.~T. Chughtai, R.~Ullah, Fractional study of the non-linear
  burgers’ equations via a semi-analytical technique, Fractal and Fractional
  7~(103) (2023) 1--17.

\bibitem{elzrps-2024N}
M.~Nadeem, Z.~Li, D.~Kumar, Y.~Alsayaad, A robust approach for computing
  solutions of fractional-order two-dimensional helmholtz equation, Scientific
  Reports volume 14~(4152) (2024) 1--13.

\bibitem{elzrps-2024}
M.~I. Liaqat, A.~Akguel, M.~Bayram, Series and closed form solution of caputo
  time-fractional wave and heat problems with the variable coefficients by a
  novel approach, Optical and Quantum Electronics 56~(203) (2024) 1--35.

\bibitem{abrps2022}
M.~I. Liaqat, S.~Etemad, S.~Rezapour, C.~Park, A novel analytical aboodh
  residual power series method for solving linear and nonlinear time-fractional
  partial differential equations with variable coefficients, Mathematics 7~(9)
  (2022) 16917--16948.

\bibitem{BSabrps2023}
M.~I. Liaqat, A.~Akgül, H.~Abu-Zinadah, Analytical investigation of some
  time-fractional black–scholes models by the aboodh residual power series
  method, Mathematics 11~(2) (2023) 1--19.

\bibitem{abrps2024}
A.~S. Alshehry, H.~Yasmin, R.~Shah, A.~Ali, I.~Khan, Fractional-order view
  analysis of fisher’s and foam drainage equations within aboodh transform,
  Engineering Computations (2024).
\newblock \href {https://doi.org/https://doi.org/10.1108/EC-08-2023-0475}
  {\path{doi:https://doi.org/10.1108/EC-08-2023-0475}}.

\bibitem{abrps2024KDV}
Y.~Jawarneh, Z.~Alsheekhhussain, M.~M. Al-Sawalha, Fractional view analysis
  system of korteweg–de vries equations using an analytical method, Fractal
  and Fractional 8~(40) (2024) 1--33.

\bibitem{sumurps2022}
V.~P. Dubey, J.~Singh, A.~M. Alshehri, S.~Dubey, D.~Kumar, Forecasting the
  behavior of fractional order bloch equations appearing in nmr flow via a
  hybrid computational technique, Chaos, Solitons and Fractals 164 (2022)
  112691.

\bibitem{robust2024}
S.~K. Khirsariya, J.~P. Chauhan, S.~B. Rao, A robust computational analysis of
  residual power series involving general transform to solve fractional
  differential equations, Mathematics and Computers in Simulation 216 (2024)
  168*186.

\bibitem{book1}
K.~Oldham, J.~Spanier, The Fractional Calculus: Theory and Applications of
  Differentiation and Integration to Arbitrary Oders, Academic Press, 1974.

\bibitem{book2}
K.~Miller, B.~Ross, An Introduction to Fractional Calculus and Fractional
  Differential Equations, Wiley, 1993.

\bibitem{book3}
J.~Podlubny, Fractional Differential Equations, Academic Press, 1999.

\bibitem{book4}
A.~Kilbas, H.~Srivastava, J.~Trujillo, Theory and Applications of Fractional
  Differential Equations, Elsevier, 2006.

\bibitem{BSrpsm2019}
Y.~Zhang, A.~Kumar, S.~Kumar, D.~Baleanu, X.~J. Yang, Residual power series
  method for time-fractional schr{\o}dinger equations, J. Nonlinear Sci. Appl
  9~(11) (2016) 5821--5829.

\bibitem{lrpsm2021soliton}
A.~El-Ajou, Adapting the laplace transform to create solitary solutions for the
  nonlinear time-fractional dispersive pdes via a new approach, Eur. Phys. J.
  Plus 136~(229) (2021) 1--15.

\end{thebibliography}

\end{document}